\documentclass[a4paper,12pt]{article}

\usepackage{amsmath,amssymb,amsthm}
\usepackage{graphics}
\usepackage{latexsym}
\usepackage{graphicx}
\usepackage{psfrag}

\theoremstyle{plain}
\newtheorem{Thm}{Theorem}[section]

\title{The phylogeny graphs of doubly partial orders}

\author{\begin{tabular}{c}
{\sc Boram PARK} \\
[1ex]
{\small 
National Institute for Mathematical Sciences} \\
{\small 
Daejeon 305-811, Korea} \\
{\small 
{\tt borampark@nims.re.kr}} \\
\\
{\sc Yoshio SANO}
\thanks{Corresponding author.} \\
[1ex]
{\small
Division of Information Engineering} \\
{\small
Faculty of Engineering, Information and Systems} \\
{\small 
University of Tsukuba, 
Ibaraki 305-8573, Japan} \\
{\small
{\tt sano@cs.tsukuba.ac.jp}}\\
\end{tabular} }

\date{}

\begin{document}

\maketitle

\begin{abstract}
The competition graph of a doubly partial order is
known to be an interval graph.
The CCE graph and the niche graph of a doubly partial order
are also known to be interval graphs
if the graphs do not contain a cycle of length four and three
as an induced subgraph, respectively.
Phylogeny graphs are variant of competition graphs.
The phylogeny graph $P(D)$ of a digraph $D$ is
the (simple undirected) graph defined by $V(P(D)):=V(D)$ and
$E(P(D)):=\{xy \mid N^+_D(x) \cap N^+_D(y) \neq \emptyset \}
\cup \{xy \mid (x,y) \in A(D) \}$,
where $N^+_D(x):=\{v \in V(D) \mid (x,v) \in A(D)\}$.

In this note, we show that the phylogeny graph of
a doubly partial order is an interval graph. 
We also show that, 
for any interval graph $G$,
there exists an interval graph $\tilde{G}$ 
such that $\tilde{G}$ contains the graph $G$ as an induced subgraph
and that $\tilde{G}$ is the phylogeny graph of a doubly partial order.


\end{abstract}

\section{Introduction}

Throughout this note,
all graphs and all digraphs are finite and simple.
The notion of competition graphs was introduced by Cohen~\cite{cohen1}
in a connection with a problem in ecology.
Given a digraph $D$, if $(u,v)$ is an arc of $D$, then
we call $v$ a {\em prey} of $u$ and $u$ a {\em predator} of $v$.
The {\em competition graph} $C(D)$ of a digraph $D$
is the graph which has the same vertex set as $D$
and has an edge between two distinct vertices $u$ and $v$
if and only if there exists a common prey of $u$ and $v$ in $D$.
Since Cohen introduced the notion of competition graphs,
several variants have been defined and studied by many authors
(see the survey articles by
Kim~\cite{Kim93} and Lundgren~\cite{Lundgren89}).
For example,
Scott~\cite{sc} introduced
competition-common enemy graphs (or CCE graphs), 
Cable, Jones, Lundgren, and Seager~\cite{cable} 
introduced niche graphs, 
and Sonntag and Teichert introduced 
competition hypergraphs~\cite{ST04} 
(see also~\cite{S09} for competition multihypergraphs 
and \cite{PS11} for the hypercompetition numbers of hypergraphs). 
As another variant,
Roberts and Sheng~\cite{Bolyai, RS97, RS98, RS2000}
introduced phylogeny graphs.
The {\em phylogeny graph} of a digraph $D$
is the graph which
has the same vertex set as $D$ and
has an edge between two distinct vertices $u$ and $v$ if and only if
there exists an arc from $u$ to $v$ or an arc from $v$ to $u$
or a common prey of $u$ and $v$ in $D$.

In the study of competition graphs and their variants,
one of important problems
is characterizing the competition graphs of 
interesting classes of digraphs 
(see \cite{fi, Fraughnaugh,kimrob,LK,lunmayras,S}
for some studies on this research direction).
In this note, we study the phylogeny graphs of doubly partial orders.
Doubly partial orders are digraphs defined as follows.
For a point $x$ in $\mathbb{R}^2$, we denote its first
and second coordinates by $x_1$ and $x_2$, respectively.
For $x, y \in \mathbb{R}^2$, we write
$x \prec y$ if $x_1 < y_1$ and $x_2 < y_2$.
A digraph $D$ is called a \emph{doubly partial order}
if there exists a finite subset $V$ of $\mathbb{R}^2$ such that
$V(D) = V$
and
$A(D) = \{(x,v) \mid v,x\in V, v \prec x \}$.
Note that,
by definition,
the out-neighborhood
$N^+_D(x) := \{v \in V(D) \mid (x,v) \in A(D)\}$
of a vertex $x$ in a doubly partial order $D$
is the set $\{v \in V(D) \mid v \prec x\}$.

We recall some results on variants of competition graphs 
of doubly partial orders.
A graph $G$ is called an \emph{interval graph}
if there exists an assignment $J:V(G) \to 2^{\mathbb{R}}$ of
real closed intervals $J(v)$ to the vertices $v$ of $G$ 
such that, for any two distinct vertices $v$ and $w$,
$vw \in E(G)$
if and only if
$J(v) \cap J(w) \neq \emptyset$.
In 2005, Cho and Kim \cite{chokim} showed the following:

\begin{Thm}[\cite{chokim}]\label{thm:CK}
The competition graph of a doubly partial order is
an interval graph.
\end{Thm}

\begin{Thm}[\cite{chokim}]\label{dpo2}
An interval graph can be made into the competition graph of
a doubly partial order by adding sufficiently many isolated vertices.
\end{Thm}

\noindent
Here, we mention that Wu and Lu \cite{WL10} 
gave a further result on this subject. 
They showed that 
a graph is the competition graph of a doubly partial order 
if and only if it is an interval graph, at least
half of whose maximal cliques are isolated vertices. 

After the study of Cho and Kim, 
several authors studied variants of 
competition graphs of doubly partial orders. 
In 2007, Kim, Kim, and Rho \cite{SJkim} showed 
that 
the CCE graph of a doubly partial order is an interval graph
unless it contains $C_4$ as an induced subgraph. 
In 2009, Kim, Lee, Park, Park, and Sano \cite{KLPPS} showed
a similar result for niche graphs. 
The niche graph of a doubly partial order 
is an interval graph 
unless it contains $C_3$ as an induced subgraph. 
In 2011, Park, Lee, and Kim \cite{PLK} 
studied the $m$-step competition graphs of doubly partial orders, 
and showed that, for any positive integer $m$, 
the $m$-step competition graph of a doubly partial order 
is an interval graph. 
Recently, Kim, Lee, Park, and Sano \cite{KLPS12} 
studied the competition hypergraphs of doubly partial orders. 

As the phylogeny graph is an important variant
of the competition graph, 
it is natural to ask whether 
the phylogeny graph 
of a doubly partial order is an interval graph 
or not. 
In the following section,
we show that the phylogeny graph of a doubly partial order
is always an interval graph.

\section{Main Results}

We use the following notations in this section:
For $x, y \in \mathbb{R}^2$, 
\begin{eqnarray*}
&& x \searrow y \quad \iff \quad x_1 \le y_1 \text{ and } y_2 \le x_2 \\
&& x \wedge y := (\min\{x_1,y_1\}, \min\{x_2,y_2\}) \in \mathbb{R}^2.
\end{eqnarray*}
The following theorem is our first main result.

\begin{Thm}\label{thm:interval}
The phylogeny graph of a doubly partial order
is an interval graph.
\end{Thm}

\begin{proof}
Let $G$ be the phylogeny graph of
a doubly partial order $D$.
We shall give an interval assignment $J:V(G) \to 2^{\mathbb{R}}$
such that $G$ is the intersection graph of the family of those intervals.
We may disregard the isolated vertices of $G$
since
we can assign an interval to an isolated vertex 
that does not overlap with any other interval. 

Let $f:\mathbb{R}^2 \to \mathbb{R}$ be a function defined by
$f(x)=f(x_1,x_2):=x_2-x_1$.
For a (non-isolated) vertex $x$ of $D$,
we define
\[
J(x):=\textrm{conv} \{ f(a) \in \mathbb{R} \mid 
a \in N^+_D(x) \cup \{x\} \}
\subseteq \mathbb{R},
\]
where $\textrm{conv}(S)$ means the convex hull of a set $S$ in $\mathbb{R}$.
Note that $J(x) \neq \emptyset$
since $f(x) \in J(x)$.
We will show that the graph
obtained by deleting all the isolated vertices from $G$
is the intersection graph of the family of these intervals.

Take two adjacent vertices $x$ and $y$ of $G$.
Then it holds that $x \prec y$, or $y \prec x$, or
there exists a vertex $a$ such that $a \prec x$ and $a \prec y$.
If $x \prec y$, then $f(x) \in J(x) \cap J(y)$.
If $y \prec x$, then $f(y) \in J(x) \cap J(y)$.
If there exists a vertex $a$ such that $a \prec x$ and $a \prec y$,
then $f(a) \in J(x) \cap J(y)$.
Therefore, we have $J(x) \cap J(y) \neq \emptyset$.

Next, take two non-isolated vertices
$x$ and $y$ which are not adjacent in $G$.
Then it holds that $x \searrow y$ or $y \searrow x$.
Without loss of generality, we may assume that $x \searrow y$.
We claim that
\begin{equation}\label{eq:sep}
\min J(x) > \max J(y).
\end{equation}
Note that $x \wedge y = (x_1,y_2)$ since $x \searrow y$. 
Since $N^+_D(x) \cap N^+_D(y)=\emptyset$,
if $y_2<x_2$ then we have
\begin{eqnarray*}
N^+_D(x) 
&\subseteq& 
\{v \in \mathbb{R}^2 \mid v_1<x_1, y_2 \leq v_2 < x_2 \} \\
&\subseteq& 
\{v \in \mathbb{R}^2 \mid v_2-v_1 > y_2-x_1 \} \\
&=& 
\{v \in \mathbb{R}^2 \mid f(v) > f(x \wedge y) \}.
\end{eqnarray*}
Therefore,
\begin{equation}\label{eq:sep2}
y_2<x_2 \quad \Rightarrow \quad \min J(x) > f(x \wedge y).
\end{equation}
Similarly, if $x_1<y_1$  then we have
\begin{eqnarray*}
N^+_D(y) 
& \subseteq & 
\{v \in \mathbb{R}^2 \mid x_1 \leq v_1 <y_1, v_2 < y_2 \} \\
& \subseteq & 
\{v \in \mathbb{R}^2 \mid v_2-v_1 < y_2-x_1 \} \\
& = & 
\{v \in \mathbb{R}^2 \mid f(v) < f(x \wedge y) \}.
\end{eqnarray*}
Therefore,
\begin{equation}\label{eq:sep3}
x_1<y_1 \quad \Rightarrow \quad \max J(y) < f(x \wedge y).
\end{equation}
Suppose that $x_1=y_1$. 
Then $x \wedge y = y$. 
Since $N^+_D(y) = N^+_D(x \wedge y) \subseteq N^+_D(x)$ and
$N^+_D(x) \cap N^+_D(y)=\emptyset$,
it holds that $N^+_D(y) =\emptyset$.
Therefore $J(y) = \{ f(y) \}$. Thus we have
\begin{equation}\label{eq:sep4}
x_1=y_1 \quad \Rightarrow \quad \max J(y) = f(x \wedge y).
\end{equation}
Suppose that $x_2=y_2$. 
Then $x \wedge y = x$. 
Since $N^+_D(x) = N^+_D(x \wedge y) \subseteq N^+_D(y)$ and
$N^+_D(x) \cap N^+_D(y)=\emptyset$,
it holds that $N^+_D(x) =\emptyset$.
Therefore $J(x) = \{ f(x) \}$. Thus we have
\begin{equation}\label{eq:sep5}
x_2=y_2 \quad \Rightarrow \quad \min J(x) = f(x \wedge y).
\end{equation}
Now, we consider the three possible cases (see Figure \ref{fig:1}).
If $x_1<y_1$ and $y_2<x_2$,
then (\ref{eq:sep}) follows from (\ref{eq:sep2}) and (\ref{eq:sep3}).
If $x_1=y_1$ and $y_2<x_2$,
then (\ref{eq:sep}) follows from (\ref{eq:sep2}) and (\ref{eq:sep4}).
If $x_1<y_1$ and $x_2=y_2$,
then (\ref{eq:sep}) follows from (\ref{eq:sep3}) and (\ref{eq:sep5}).
Thus, we have $J(x) \cap J(y) = \emptyset$.

Hence the theorem holds.
\end{proof}

\begin{figure}
\psfrag{x}{$x$}
\psfrag{y}{$y$}
\psfrag{xy}{$x \wedge y$}
\psfrag{x1}{$x_1$}
\psfrag{x2}{$x_2$}
\psfrag{y1}{$y_1$}
\psfrag{y2}{$y_2$}
\begin{center}
\begin{tabular}{ccc}
\includegraphics[scale=0.4]{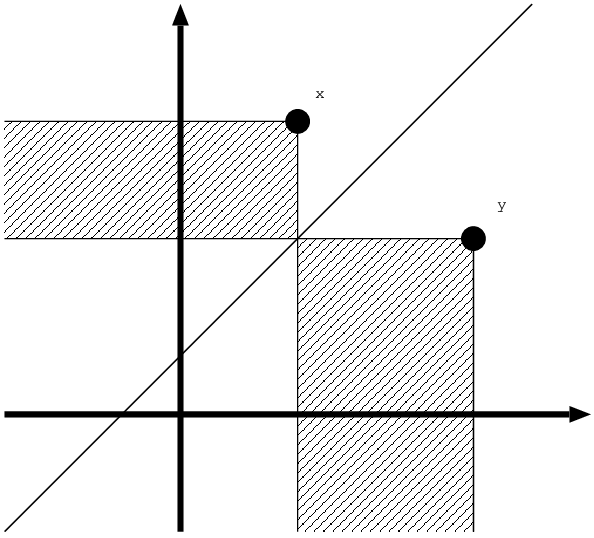} &
\includegraphics[scale=0.4]{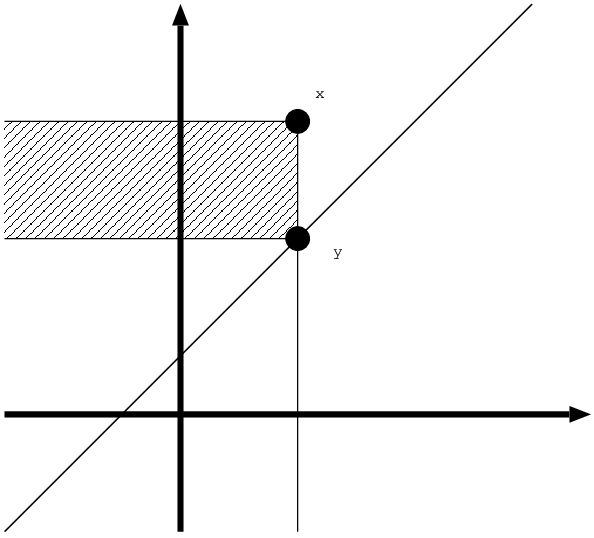} &
\includegraphics[scale=0.4]{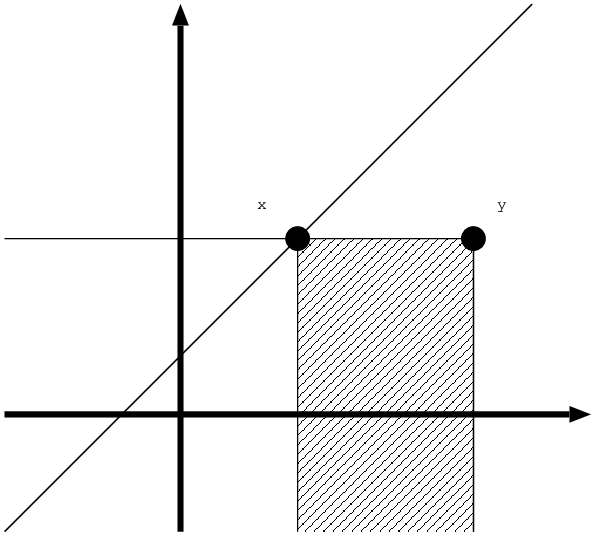}
\end{tabular}
\caption{Pictures for Proof of Theorem \ref{thm:interval}}
\label{fig:1}
\end{center}
\end{figure}

Theorem \ref{thm:CK} and Theorem \ref{dpo2} are complementary to each other. 
In the same fashion, after establishing Theorem \ref{thm:interval}, 
a natural question is if every interval graph 
is the phylogeny graph of some doubly partial order. 
However, the answer for this question is NO.
The following theorem shows that
not every interval graph $G$ has
a doubly partial order whose phylogeny graph is equal to $G$.

\begin{Thm}\label{thm:path}
Let $G$ be an interval graph. 
If $G$ has 
two adjacent vertices $u$ and $v$ of degree at least two 
such that the edge $uv$ is not contained in any triangle in $G$, 
then $G$ cannot be the phylogeny graph of a doubly partial order. 
\end{Thm}

\begin{proof}
Let $G$ be an interval graph and 
suppose that $G$ has two adjacent vertices $u$ and $v$ of degree at least two 
such that the edge $uv$ is not contained in any triangle in $G$. 
Let $a$ (resp. $b$) be a vertex of $G$ 
other than $v$ (resp. $u$) 
which is adjacent to $u$ (resp. $v$). 
Then 
$P=auvb$ is a path of length $3$. 
Suppose that there exists a doubly partial order $D$ 
such that the phylogeny graph of $D$ is equal to $G$. 
Since the edge $uv$ is not contained in any triangle in $G$,
it holds that $N_G(u) \cap N_G(v) =\emptyset$.
Since 
$N_G(u) \cap N_G(v) =\emptyset$,
there cannot exist a vertex $z$ in $D$
such that $z \prec u$ and $z \prec v$.
However, since $uv \in E(G)$, 
either
$u \prec v$ or $v \prec u$.
Without loss of generality,
we may assume that
$v \prec u$. 
Since $vb \in E(G)$, 
it holds that 
$v$ and $b$ have a common prey or 
$v \prec b$ 
or $b \prec v$. 
If there exists a vertex $z$ in $D$
such that $z \prec b$ and $z \prec v$,
then $z \prec u$ and
$b,u,v,z$ form a clique of size $4$ in $G$,
which is a contradiction to the fact that
$N_G(u) \cap N_G(v) =\emptyset$.
If $v \prec b$,
then $u$ and $b$ have $v$ as a common prey 
and so $u$ and $b$ are adjacent, which is also a contradiction 
to the fact that
$N_G(u) \cap N_G(v) =\emptyset$. 
If $b \prec v$, then $b \prec u$
and so $u$ and $b$ are adjacent,
which is a contradiction.
Therefore, in each case, we reach a contradiction.
Hence, $G$ cannot be the phylogeny graph of a doubly partial order.
\end{proof}

The above result is contrary to Theorem \ref{dpo2} stating that
we can make any interval graph into the competition graph 
of a doubly partial order by adding 
sufficiently many isolated vertices. 

Though not every interval graph is
the phylogeny graph of a doubly partial order
even if we allow us to add isolated vertices,
we can show that, given an interval graph $G$,
there exists an ``extension" of $G$
such that it is the phylogeny graph of a doubly partial order.

\begin{Thm}\label{thm:ext}
For any interval graph $G$,
there exists an interval graph $\tilde{G}$
such that $\tilde{G}$ contains the graph $G$ as an induced subgraph
and that $\tilde{G}$ is the phylogeny graph of a doubly partial order.
\end{Thm}

\begin{proof}
We may assume that $G$ has no isolated vertices.
By Theorem \ref{dpo2},
there exists a doubly partial order $D$ such that
the competition graph of $D$ is equal to 
the graph obtained from $G$ by adding 
sufficiently many isolated vertices. 
Let $\tilde{G}$ be the phylogeny graph of $D$. 
Then, by Theorem \ref{thm:interval}, 
$\tilde{G}$ is an interval graph. 
Let $S \subseteq V(D)$ be the set of vertices having a prey in $D$, 
that is,
\[
S := \{ v \in V(D) \mid N_D^{+}(v) \neq \emptyset \}.
\]
Let $H$ be the subgraph of $\tilde{G}$ induced by $S$.
We will show that $E(H)=E(G)$.

Take an edge $uv$ of $G$.
Since the competition graph of $D$ contains the graph $G$,
there exists a common prey $z$ of $u$ and $v$ in $D$.
Therefore, $u$ and $v$ are adjacent in $\tilde{G}$.
Also, since $u$ and $v$ have a prey in $D$,
it means that $u,v \in S = V(H)$.
Since $H$ is an induced subgraph of $\tilde{G}$, we have $uv\in E(H)$.
Therefore, $E(G) \subseteq E(H)$.

Take an edge $uv$ of $H$.
Since $u,v \in V(H)=S$,
we have $N_D^{+}(u) \neq \emptyset$ and $N_D^{+}(v) \neq \emptyset$.
To show that $uv$ is also an edge of $G$,
it is sufficient to show that $u$ and $v$ have a common prey in $D$.
Suppose that $u$ and $v$ do not have a common prey in $D$.
Since $N_D^{+}(u) \neq \emptyset$ and $N_D^{+}(v) \neq \emptyset$,
there exist two vertices $a$ and $b$ in $D$
such that $a \prec u$ and $b \prec v$.
Since $u$ and $v$ are adjacent in $H$
but have no common prey in $D$,
either $u \prec v$ or $v \prec u$. 
If $u \prec v$ then $a \prec v$ since $a \prec u$,
and so $a$ is a common prey of $u$ and $v$ in $D$, 
which is a contradiction.
Similarly,
if $v \prec u$ then $b \prec u$ since $b \prec v$,
and so $b$ is a common prey of $u$ and $v$ in $D$, 
which is a contradiction.
Therefore, $u$ and $v$ have a common prey in $D$.
Thus, $uv$ is an edge of $G$ and so $E(H) \subseteq E(G)$.
Hence, $E(H)=E(G)$.

Let $H_0$ be the subgraph of $\tilde{G}$
obtained by deleting all the isolated vertices from $H$.
Then $H_0$ is an induced subgraph of $\tilde{G}$ which is equal to $G$.
We complete the proof.
\end{proof}

\section{Concluding Remarks}

In this note, we showed that the phylogeny graph of a doubly partial order
is an interval graph (Theorem \ref{thm:interval})
and that any interval graph has a graph extension
which is the phylogeny graph of a doubly partial order
(Theorem \ref{thm:ext}).

By Theorem \ref{thm:ext}, for an interval graph $G$,
we can define the \emph{doubly partial order phylogeny number}
$p_{\text{dpo}}(G)$ of $G$ to be the smallest nonnegative integer $r$
such that $r:=|V(\tilde{G}) \setminus V(G)|$
where $\tilde{G}$ is an interval graph containing the graph $G$ 
as an induced subgraph
and $\tilde{G}$ is the phylogeny graph of a doubly partial order.
The doubly partial order phylogeny number of an interval graph $G$
may be different from the phylogeny number of $G$.
It can be shown by using the path $P$ of length three. 
By \cite[Theorem 7]{RS98},
if $G$ is a chordal graph, then
the phylogeny number $p(G)$ of $G$ is equal to $0$.
Therefore $p(P)=0$.
On the other hand, it follows from Theorem \ref{thm:path} that
$p_{\text{dpo}}(P) >0$.
It would be interesting to find 
the doubly partial order phylogeny numbers
of various interval graphs.


\end{document}